\documentclass{amsart}
\usepackage{graphicx, color}
\usepackage{amscd}
\usepackage{amsmath,empheq}
\usepackage{amsfonts}
\usepackage{amssymb}
\usepackage{mathrsfs}

\usepackage[all]{xy}

\newtheorem{theorem}{Theorem}

\newtheorem{prop}[theorem]{Proposition}

\newtheorem{corollary}[theorem]{Corollary}

\newenvironment{proof-sketch}{\noindent{\bf Sketch of Proof}\hspace*{1em}}{\qed\bigskip}

\everymath{\displaystyle}

\newcommand{\RR}{\mathbb R}
\newcommand{\NN}{\mathbb N}

\renewcommand{\leq}{\leqslant}

\renewcommand{\geq}{\geqslant}

\begin{document}

\title[Nodal solutions for the Robin $p$-Laplacian plus an indefinite potential]{Nodal solutions for the Robin $p$-Laplacian plus an indefinite potential and a general reaction term}

\author[N.S. Papageorgiou]{Nikolaos S. Papageorgiou}
\address[N.S. Papageorgiou]{National Technical University, Department of Mathematics,
				Zografou Campus, Athens 15780, Greece}
\email{\tt npapg@math.ntua.gr}

\author[V.D. R\u{a}dulescu]{Vicen\c{t}iu D. R\u{a}dulescu}
\address[V.D. R\u{a}dulescu]{Department of Mathematics, Faculty of Sciences,
King Abdulaziz University, P.O. Box 80203, Jeddah 21589, Saudi Arabia \& Department of Mathematics, University of Craiova,
          200585 Craiova, Romania}
\email{\tt vicentiu.radulescu@imar.ro}

\author[D.D. Repov\v{s}]{Du\v{s}an D. Repov\v{s}}
\address[D.D. Repov\v{s}]{Faculty of Education and Faculty of Mathematics and Physics,
University of Ljubljana, SI-1000 Ljubljana, Slovenia}\email{dusan.repovs@guest.arnes.si}


\begin{abstract}
We consider a nonlinear Robin problem driven by the $p$-Laplacian plus an indefinite potential. The reaction term is of arbitrary growth and only conditions near zero are imposed. Using critical point theory together with suitable truncation and perturbation techniques and comparison principles, we show that the problem admits a sequence of distinct smooth nodal solutions converging to zero in $C^1(\overline{\Omega})$.
\end{abstract}

\maketitle

\section{Introduction}

Let $\Omega\subseteq\RR^N$ be a bounded domain with a $C^2$-boundary $\partial\Omega$. In this paper we study the following nonlinear Robin problem
\begin{equation}\label{eq1}
	\left\{\begin{array}{ll}
		-\Delta_p u(z)+\xi(z)|u(z)|^{p-2}u(z)=f(z,u(z))\ \mbox{in}\ \Omega,\\
		\frac{\partial u}{\partial n_p}+\beta(z)|u|^{p-2}u=0\ \mbox{on}\ \partial\Omega.
	\end{array}\right.
\end{equation}

In this problem, $\Delta_p$ denotes the $p$-Laplace differential operator defined by
$$\Delta_p u={\rm div}\,(|Du|^{p-2}Du)\ \mbox{for all}\ u\in W^{1,p}(\Omega),\ 1<p<\infty.$$

The potential function $\xi\in L^{\infty}(\Omega)$ is indefinite (that is, sign changing) and the reaction term $f(z,x)$ is a Carath\'eodory function (that is, for all $x\in\RR$, the mapping $z\mapsto f(z,x)$ is measurable and for almost all $z\in\Omega$,  $x\mapsto f(z,x)$ is continuous). We do not impose any global polynomial growth condition on $f(z,\cdot)$. All the conditions on $f(z,\cdot)$ concern its behaviour near zero. In the boundary condition, $\frac{\partial u}{\partial n_p}$ denotes the generalized normal derivative defined by extension of the map
$$C^1(\overline{\Omega})\ \ni\ u \mapsto |Du|^{p-2}(Du,n)_{\RR^N},$$
with $n(\cdot)$ being the outward unit normal on $\partial\Omega$. The boundary coefficient $\beta\in C^{0,\alpha}(\partial\Omega)$ (with $0<\alpha<1$) satisfies $\beta(z)\geq 0$ for all $z\in\partial\Omega$. When $\beta=0$, we have the usual Neumann problem.

Using variational methods, together with suitable truncation and perturbation techniques and comparison principles, and an abstract result of Kajikiya \cite{5}, we show that the problem admits an infinity of smooth nodal (that is, sign changing) solutions converging to zero in $C^1(\overline{\Omega})$. Our starting point is the recent work of Papageorgiou and R\u adulescu \cite{9}, where the authors produced an infinity of nodal solutions for a nonlinear Robin problem with zero potential (that is, $\xi\equiv0$) and a reaction term of arbitrary growth. They assumed that the reaction term $f(z,x)$ is a Carath\'eodory function and there exists $\eta>0$ such that for almost all $z\in\Omega$, $f(z,\cdot)|_{[-\eta,\eta]}$ is odd and $f(z,\eta)\leq 0\leq f(z,-\eta)$ (the second inequality follows from the first inequality and the oddness of $f(z,\cdot)$). Moreover, they assumed that for almost all $z\in\Omega,\ f(z,\cdot)$ exhibits a concave (that is, a strictly $(p-1)$-superlinear)
  term near zero. So, $f(z,\cdot)$ has zeros of constant sign and it presents a kind of oscillatory behaviour near zero. In the present work  we introduce in the equation an indefinite potential term $\xi(z)|x|^{p-2}x$ and we remove the requirement that $f(z,\eta)\leq 0$ for almost all $z\in\Omega$. We point out that this was a very convenient hypothesis, since the constant function $\tilde{u}\equiv\eta>0$ provided an upper solution for the problem and $\tilde{v}=-\eta<0$ a lower solution. With them, the analysis of problem (\ref{eq1}) was  significantly simplified. The absence of this condition in the present work, changes the geometry of the problem and so we need a different approach. We should mention that in Papageorgiou and R\u adulescu \cite{9}, the differential operator is more general and is nonhomogeneous. It is an interesting open problem whether our present work  can be extended to equations driven by nonhomogeneous differential operators, as in \cite{9}.

Wang \cite{11} was the first to produce an infinity of solutions for problems with a reaction of arbitrary growth. He used cut-off techniques to study semilinear Dirichlet problems with zero potential driven by the Laplacian. More recently, Li and Wang \cite{7} produced infinitely many nodal solutions for semilinear Schr\"{o}dinger equations.
We also refer to our recent papers \cite{prr1, prr2}, which deal with the qualitative analysis of nonlinear Robin problems.

\section{Mathematical Background}

In the analysis of problem (\ref{eq1}) we will use the Sobolev space $W^{1,p}(\Omega)$, the Banach space $C^1(\overline{\Omega})$ and the ``boundary" Lebesgue spaces $L^s(\partial\Omega)$, $1\leqslant s\leqslant+\infty$.

We denote by $\|\cdot\|$  the norm of the Sobolev space $W^{1,p}(\Omega)$ defined by
$$\|u\|=\left[\|u\|^p_p + \|Du\|_p^p\right]^{^1/_p}\ \mbox{for all}\ u\in W^{1,p}(\Omega).$$

The Banach space $C^1(\overline\Omega)$ is an ordered Banach space with positive (order) cone
$$C_+=\{u\in C^1(\overline{\Omega}) : u(z)\geqslant0\ \mbox{for all}\ z\in\overline{\Omega}\}.$$

This cone has a nonempty interior given by
$$D_+=\{u\in C_+ : u(z)>0\ \mbox{for all}\ z\in\overline{\Omega}\}.$$

Also, let $\hat{D}_+\subseteq C_+$ be defined by
$$\hat{D}_+ = \{u\in C_+ : u(z)>0\ \mbox{for all}\ z\in\Omega, \frac{\partial u}{\partial n}|_{\partial\Omega\cap u^{-1}(0)}<0\ \mbox{if}\ \partial\Omega\cap u^{-1}(0)\neq\emptyset\}.$$

Evidently, $\hat{D}_+\subseteq C^1(\overline{\Omega})$ is open and $D_+\subseteq \hat{D}_+$.

On $\partial\Omega$ we consider the $(N-1)$-dimensional Hausdorff (surface) measure $\sigma(\cdot)$. Using this measure, we can define in the usual way the boundary Lebesgue spaces $L^s(\partial\Omega)$, $1\leqslant s\leqslant\infty$. From the theory of Sobolev spaces, we know that there exists a unique continuous linear map $\gamma_0:W^{1,p}(\Omega)\rightarrow L^p(\partial\Omega)$, known an the ``trace operator", such that
$$\gamma_0(u)=u|_{\partial\Omega}\ \mbox{for all}\ u\in W^{1,p}(\Omega)\cap C^1(\overline{\Omega}).$$

So, the trace operator assigns ``boundary values" to every Sobolev function. The trace operator is compact into $L^s(\partial\Omega)$ for all $s\in\left[1,\frac{(N-1)p}{N-p}\right)$ if $p<N$ and into $L^s(\partial\Omega)$ for all $s\geq 1$ if $N\leq p$. Moreover, we have
$${\rm ker}\,\gamma_0=W^{1,p}_0(\Omega)\ \mbox{and}\ {\rm im}\,\gamma_0=W^{\frac{1}{p'},p}(\partial\Omega)\quad(\frac{1}{p} + \frac{1}{p'} = 1).$$

In the sequel, for the sake of notational simplicity, we will drop the use of operator $\gamma_0$. All restrictions of Sobolev functions on $\partial\Omega$, are understood in the sense of traces.

Given $h_1,h_2\in L^\infty(\Omega)$, we write that $h_1\prec h_2$ if and only if for every compact set $K\subseteq\Omega$, we can find $\epsilon=\epsilon(K)>0$ such that
$$h_1(z)+\epsilon\leqslant h_2(z)\ \mbox{for almost all}\ z\in K.$$

We see that, if $h_1, h_2\in C(\Omega)$ and $h_1(z)< h_2(z)$ for all $z\in\Omega$, then $h_1\prec h_2.$

The next strong comparison theorem can be found in Fragnelli, Mugnai and Papageorgiou \cite{3}.
\begin{prop}\label{prop1}
	Assume that $\tilde{\xi}, h_1, h_2 \in L^\infty(\Omega)$, $\tilde{\xi}(z)\geq 0$, for almost all $z\in\Omega$, $h_1\prec h_2$, $u\in C^1(\overline{\Omega})$, $u\neq0$, $v\in D_+$, $u\leq v$ and that they satisfy
	\begin{eqnarray*}
		\begin{array}{ll}
			& (i)\ -\Delta_p u(z)+\tilde{\xi}(z)|u(z)|^{p-2}u(z) = h_1(z)\ \mbox{for almost all}\ z\in\Omega;\\
			& (ii)\ -\Delta_p v(z)+\tilde{\xi}(z)v(z)^{p-1} = h_2(z)\ \mbox{for almost all}\ z\in\Omega;\ \mbox{and}\\
			& (iii)\ \frac{\partial v}{\partial n}|_{\partial\Omega}<0.
		\end{array}
	\end{eqnarray*}
	Then $v-u \in \hat{D}_+$.
\end{prop}

As we have already mentioned in the introduction, the sequence of nodal solutions will be generated by using an abstract result of Kajikiya \cite{5}, which is essentially an extension of the symmetric mountain pass theorem (see also Wang \cite{11}). Recall that, if $X$ is a Banach space and $\varphi\in C^1(X,\RR)$, we say that $\varphi$ satisfies the ``Palais-Smale condition" (``PS-condition", for short), if the following holds:

``Every sequence $\{u_n\}_{n\geq 1}\subseteq X$ such that $\{\varphi(u_n) \}_{n\geq 1}\subseteq \RR$ is bounded and
$$\varphi'(u_n)\rightarrow0\ \mbox{in}\ X^*\ \mbox{as}\ n\rightarrow\infty,$$

\ admits a strongly convergent subsequence".
\begin{theorem}\label{th2}
	Let $X$ be a Banach space and suppose that $\varphi\in C^1(X,\RR)$ satisfies the $PS$-condition,  is even and bounded below, $\varphi(0)=0$, and for every $n\in\NN$ there exist a nontrivial finite dimensional subspace $V_n$ of $X$ and $\rho_n>0$ such that
	$$\sup[\varphi(u):u\in V_n\cap \partial B_{\rho_n} ]<0,$$
	where $\partial B_{\rho_n}=\{u\in X:||u||=\rho_n\}$.

Then there exists a sequence $\{u_n\}_{n\geq1}\subseteq X$ such that
	\begin{itemize}
		\item (i)\ $\varphi'(u_n)=0$ for all $n\in\NN$ (that is, $u_n$ is a critical point of $\varphi$);\ and
		\item (ii)\ $\varphi(u_n)<0$ for all $n\in\NN$ and $u_n\rightarrow0$ in $X$.
	\end{itemize}
\end{theorem}

In what follows, we denote by $A:W^{1,p}(\Omega)\rightarrow W^{1,p}(\Omega)^*$  the nonlinear map defined by
$$\langle A(u),h\rangle = \int_{\Omega}|Du|^{p-2}(Du,Dh)_{\RR^\NN}dz\ \mbox{for all}\ u,h\in W^{1,p}(\Omega).$$

It is well-known (see, for example, Gasinski and Papageorgiou \cite{4}), that $A(\cdot)$ is monotone continuous and of type $(S)_+$ (that is, if $u_n\xrightarrow[]{\text{w}}u$ in $W^{1,p}(\Omega)$ and $\limsup\limits_{n\rightarrow\infty}\langle A(u_n),u_n-u\rangle\leq 0$, then $u_n\rightarrow u$ in $W^{1,p}(\Omega)$).

For $x\in\RR$, we set $x^\pm=\max\{\pm x,0\}$. Then, given $u\in W^{1,p}(\Omega)$, we can define
$$u^\pm(\cdot)=u(\cdot)^\pm.$$

We know that $u^\pm\in W^{1,p}(\Omega),\ |u|=u^++u^-$, and $ u=u^+-u^-.$

Finally, if $X$ is a Banach space and $\varphi\in C^1(X,\RR)$, then
$$K_\varphi=\{u\in X:\varphi'(u)=0 \}$$
is the critical set of $\varphi$.

\section{Infinitely Many Nodal Solutions}

In this section we prove our main result, namely the existence of a whole sequence of distinct nodal solutions $\{u_n\}_{n\geq1}$ which converge to zero in $C^1(\overline{\Omega})$.

Our hypotheses on the data of problem (\ref{eq1}) are the following:
\begin{itemize}
	\item [$H(\xi):$] $\xi\in L^\infty(\Omega)$.
	\item [$H(\beta):$] $\beta\in C^{0,\alpha}(\partial\Omega)$ with $\alpha\in(0,1)$ and $\beta(z)\geq0$ for all $z\in\partial\Omega$.
	\item [$H(f):$] $f:\Omega\times\RR\rightarrow\RR$ is a Carath\'eodory function such that for almost all $z\in\Omega,\ f(z,0)=0, f(z,\cdot)$ is odd on $[-\eta,\eta]$ with $\eta>0$ and the following conditions hold:
			\begin{itemize}
				\item [(i)] there exists $a_\eta\in L^\infty(\Omega)$ such that
						$$|f(z,x)|\leq a_\eta(z)\ \mbox{for almost all}\ z\in\Omega,\ \mbox{all}\ |x|\leq\eta;$$
				\item [(ii)] $\lim\limits_{x\rightarrow0}\frac{f(z,x)}{|x|^{p-2}x}=+\infty$ uniformly for almost all $z\in\Omega$; and
				\item [(iii)] there exists $\tilde{\vartheta}>0$ such that for almost all $z\in\Omega$ the function
						$$x\mapsto f(z,x)+\tilde{\vartheta}|x|^{p-2}x$$
						is nondecreasing on $[-\eta,\eta].$
			\end{itemize}
\end{itemize}

Given $\hat{\vartheta}\in (0,||\xi||_\infty]$ and $r>p$, using hypotheses $H(f)$ above, we  can find $c_1=c_1(\hat{\vartheta},r)>0$ such that
\begin{equation}\label{eq2}
	f(z,x)x\geq\hat{\vartheta}|x|^p-c_1|x|^r\ \mbox{for almost all}\ z\in\overline{\Omega},\ \mbox{all}\ |x|\leq\eta.
\end{equation}

Let $\vartheta\in(0,\hat{\vartheta})$ and introduce the following Carath\'eodory function
\begin{equation}\label{eq3}
	k(z,x)=\left\{\begin{array}{ll}
		-\vartheta\eta^{p-1}+c_1\eta^{r-1} & \mbox{if}\ x<-\eta \\
		 \vartheta|x|^{p-2}x-c_1|x|^{r-2}x & \mbox{if}\ -\eta\leq x \leq \eta \\
		 \vartheta\eta^{p-1}-c_1\eta^{r-1} & \mbox{if}\ \eta<x	.
	\end{array}\right.
\end{equation}

We consider the following auxiliary Robin problem:
\begin{equation}\label{eq4}
	\left\{ \begin{array}{l}
		-\Delta_p u(z)+||\xi||_\infty|u(z)|^{p-2}u(z)=k(z,u(z))\ \mbox{in}\ \Omega,\\
		\frac{\partial u}{\partial n_p}+\beta(z)|u|^{p-2}u=0\ \mbox{on}\ \partial\Omega.
	\end{array} \right.
\end{equation}
\begin{prop}\label{prop3}
	If hypotheses $H(\xi), H(\beta)$ hold, then problem (\ref{eq4}) admits a unique positive solution $\tilde{u}\in[0,\eta]\cap D_+$ and since problem (\ref{eq4}) is odd, $\tilde{v}=-\tilde{u}\in[-\eta,0]\cap(-D_+)
	$ is the unique negative solution of (\ref{eq4}).
\end{prop}
\begin{proof}
	First, we establish the existence of a positive solution.
	
	To this end, let $\hat{\xi}_0>||\xi||_\infty$ and consider the following Carath\'eodory function
	\begin{equation}\label{eq5}
		\hat{k}(z,x)=\left\{\begin{array}{ll}
			-k(z,\eta)-\hat{\xi}_0\eta^{p-1} & \mbox{if}\ x<-\eta \\
			k(z,x)+\hat{\xi}_0|x|^{p-2}x		 & \mbox{if}\ -\eta\leq x \leq\eta \\
			k(z,\eta)+\hat{\xi}_0\eta^{p-1}  & \mbox{if}\ \eta< x.
		\end{array}\right.
	\end{equation}
	
	We set $\hat{K}(z,x)=\int^x_0\hat{k}(z,s)ds$ and consider the $C^1$-functional $\hat{\psi}_+:W^{1,p}(\Omega)\rightarrow\RR$ defined by
	$$\hat{\psi}_+(u)=\frac{1}{p}||Du||^p_p+\frac{1}{p}\int_\Omega\left[||\xi||_\infty+\hat{\xi}_0\right]|u|^pdz+\frac{1}{p}\int_{\partial\Omega}\beta(z)|u|^pd\sigma-\int_\Omega \hat{K}(z,u^+)dz$$
	for all $u\in W^{1,p}(\Omega)$.
	
	From (\ref{eq5}) it is clear that $\hat{\psi}_+$ is coercive. Also, using the Sobolev embedding theorem and the compactness of the trace map, we see that $\hat{\psi}_+$ is sequentially weakly lower semicontinuous. So, by the Weierstrass-Tonelli theorem, we can find $\tilde{u}\in W^{1,p}(\Omega)$ such that
	\begin{equation}\label{eq6}
		\hat{\psi}_+(\tilde{u})=\inf\left[\hat{\psi}_+(u):u\in W^{1,p}(\Omega)\right].
	\end{equation}
	
	Hypothesis $H(f)(ii)$ implies that given any $\mu>0$, we can find $\delta=\delta(\mu)>0$ such that
	\begin{equation}\label{eq7}
		F(z,x)\geq\frac{\mu}{p}|x|^p\ \mbox{for almost all}\ z\in\Omega,\ \mbox{all}\ |x|\leq\delta.
	\end{equation}
	
	Let $\hat{u}_1$ be the $L^p$-normalized positive principal eigenfunction of the operator $-\Delta_p+\xi I$ with Robin boundary condition, corresponding to the first eigenvalue $\hat{\lambda}_1\in\RR$. From Papageorgiou and R\u adulescu \cite{8} we know that $\hat{u}_1\in D_+$ and of course, $||\hat{u}_1||_p=1$. So, we can choose a small $t\in(0,1)$  such that
\begin{equation}\label{eq8}
	t\hat{u}_1(z)\in\left(0,\delta\right]\ \mbox{for all}\ z\in\overline{\Omega}.
\end{equation}

We have
$$\hat{\psi}_+(t\hat{u}_1)\leq\frac{t^p}{p}\left[\hat{\lambda}_1+\hat{\xi}_0-\mu\right]\ (\mbox{use (\ref{eq7}), (\ref{eq8}) and recall that}\ ||\hat{u}_1||_p=1).$$

We choose $\mu>\hat{\lambda}_1+\hat{\xi}_0$ and obtain
\begin{eqnarray*}
	&&\hat{\psi}_+(t\hat{u}_1)<0,\\
	&\Rightarrow&\hat{\psi}_+(\tilde{u})<0=\hat{\psi}_+(0)\ (\mbox{see (\ref{eq6})}),\\
	&\Rightarrow&\tilde{u}\neq 0.
\end{eqnarray*}

By (\ref{eq6}) we have
\begin{eqnarray}\label{eq9}
	&&\hat{\psi}_+'(\tilde{u})=0,\nonumber\\
	&\Rightarrow&\left\langle A(\tilde{u}),h\right\rangle+\int_{\Omega}[||\xi||_{\infty}+\hat{\xi}_0]|\tilde{u}|^{p-2}\tilde{u}hdz+\int_{\partial\Omega}\beta(z)|\tilde{u}|^{p-2}\tilde{u}hd\sigma=\nonumber\\
	&&\int_{\Omega}\hat{k}(z,\tilde{u}^+)hdz\ \mbox{for all}\ h\in W^{1,p}(\Omega).
\end{eqnarray}

Choosing $h=-\tilde{u}^-\in W^{1,p}(\Omega)$ in (\ref{eq9}), we have
\begin{eqnarray*}
	&&||D\tilde{u}^-||^p_p+c_2||\tilde{u}^-||^p_p\leq 0\ \mbox{for some}\ c_2>0\\
	&&(\mbox{recall that}\ \hat{\xi}_0>0\ \mbox{and use hypothesis}\ H(\beta))\\
	&\Rightarrow&\tilde{u}\geq 0,\ \tilde{u}\neq 0.
\end{eqnarray*}

Also, in (\ref{eq9}) we choose $h=(\tilde{u}-\eta)^+\in W^{1,p}(\Omega)$. Then
\begin{eqnarray*}
	&&\left\langle A(\tilde{u}),(\tilde{u}-\eta)^+\right\rangle+\int_{\Omega}[||\xi||_{\infty}+\hat{\xi}_0]\tilde{u}^{p-1}(\tilde{u}-\eta)^+dz+\int_{\partial\Omega}\beta(z)\tilde{u}^{p-1}(\tilde{u}-\eta)^+d\sigma\\
	&=&\int_{\Omega}[k(z,\eta)+\hat{\xi}_0\eta^{p-1}](\tilde{u}-\eta)^+dz\ (\mbox{see } (\ref{eq5}))\\
	&=&\int_{\Omega}[(\vartheta+\hat{\xi}_0)\eta^{p-1}-c_1\eta^{r-1}](\tilde{u}-\eta)^+dz\ (\mbox{see } (\ref{eq3}))\\
	&\leq&\int_{\Omega}[(||\xi||_{\infty}+\hat{\xi}_0)\eta^{p-1}-c_1\eta^{r-1}](\tilde{u}-\eta)^+dz\ (\mbox{recall that}\ 0<\vartheta<||\xi||_{\infty})\\
	&\leq&\left\langle A(\eta),(\tilde{u}-\eta)^+\right\rangle+\int_{\Omega}[||\xi||_{\infty}+\hat{\xi}_0]\eta^{p-1}(\tilde{u}-\eta)^+dz+\int_{\partial\Omega}\beta(z)\tilde{u}^{p-1}(\tilde{u}-\eta)^+d\sigma\\
	&&(\mbox{note that}\ A(\eta)=0\ \mbox{and use hypothesis}\ H(\beta))\\
	\Rightarrow&&\left\langle A(\tilde{u})-A(\eta),(\tilde{u}-\eta)^+\right\rangle+\int_{\Omega}[||\xi||_{\infty}+\hat{\xi}_0](\tilde{u}^{p-1}-\eta^{p-1})(\tilde{u}-\eta)^+dz\leq 0,\\
	\Rightarrow&&\tilde{u}\leq\eta\ (\mbox{recall that}\ \hat{\xi}_0>0).
\end{eqnarray*}

So, we have proved that
\begin{equation}\label{eq10}
	\tilde{u}\in[0,\eta]=\{u\in W^{1,p}(\Omega):0\leq u(z)\leq\eta\ \mbox{for almost all}\ z\in\Omega\}.
\end{equation}

It follows from (\ref{eq5}), (\ref{eq9}) and (\ref{eq10}) that
\begin{eqnarray}\label{eq11}
	&&\left\langle A(\tilde{u}),h\right\rangle+\int_{\Omega}||\xi||_{\infty}\tilde{u}^{p-1}hdz+\int_{\partial\Omega}\beta(z)\tilde{u}^{p-1}hd\sigma=\int_{\Omega}k(z,\tilde{u})hdz\nonumber\\
	&&\mbox{for all}\ h\in W^{1,p}(\Omega),\nonumber\\
	&\Rightarrow&-\Delta_p\tilde{u}(z)+||\xi||_{\infty}\tilde{u}(z)^{p-1}=k(z,\tilde{u}(z))\ \mbox{for almost all}\ z\in\Omega,\nonumber\\
	&&\frac{\partial \tilde{u}}{\partial n_p}+\beta(z)\tilde{u}^{p-1}=0\ \mbox{on}\ \partial\Omega\ (\mbox{see Papageorgiou and R\u adulescu \cite{8}}).
\end{eqnarray}

By virtue of (\ref{eq11}) and Papageorgiou and R\u adulescu \cite{10}, we have
$$\tilde{u}\in L^{\infty}(\Omega).$$

So, we can apply Theorem 2 of Lieberman \cite{6} and infer that
$$\tilde{u}\in C_+\backslash\{0\}.$$

It follows from  (\ref{eq3}), (\ref{eq5}), (\ref{eq10}) and (\ref{eq11}) that
\begin{eqnarray*}
	&&\Delta_p\tilde{u}(z)\leq[||\xi||_{\infty}+c_1||\tilde{u}||^{r-p}_{\infty}]\tilde{u}(z)^{p-1}\ \mbox{for almost all}\ z\in\Omega\\
	&\Rightarrow&\tilde{u}\in D_+,\ \mbox{that is,}\ \tilde{u}\in[0,\eta]\cap D_+
\end{eqnarray*}
by the nonlinear maximum principle (see Gasinski and Papageorgiou \cite[p. 738]{4}).

The uniqueness of this positive solution of problem (\ref{eq4}) follows from Theorem 1 of Diaz and Saa \cite{1}.

Since problem (\ref{eq4}) is odd (note that $k(z,\cdot)$ is odd, see (\ref{eq3})), it follows that
$$\tilde{v}=-\tilde{u}\in[-\eta,0]\cap(-D_+)$$
is the unique negative solution of (\ref{eq4}).
\end{proof}

Using the two constant sign solutions of problem (\ref{eq4}) produced by Proposition \ref{prop3}, we introduce the following truncation-perturbation of the reaction term $f(z,\cdot)$ (recall that $\hat{\xi}_0>||\xi||_{\infty}$)
\begin{equation}\label{eq12}
	\hat{f}(z,x)=\left\{\begin{array}{ll}
		f(z,\tilde{v}(z))+\hat{\xi}_0|\tilde{v}(z)|^{p-2}\tilde{v}(z)&\mbox{if}\ x<\tilde{v}(z)\\
		f(z,x)+\hat{\xi}_0|x|^{p-2}x&\mbox{if}\ \tilde{v}(z)\leq x\leq\tilde{u}(z)\\
		f(z,\tilde{u}(z))+\hat{\xi}_0\tilde{u}(z)^{p-1}&\mbox{if}\ \tilde{u}(z)<x.
	\end{array}\right.
\end{equation}

This is a Carath\'eodory function. We also consider the positive and negative truncations of $\hat{f}(z,\cdot)$, that is, the Carath\'eodory functions
$$\hat{f}_{\pm}(z,x)=f(z,\pm x^{\pm})\ \mbox{for all}\ (z,x)\in\Omega\times\RR.$$

We set $\hat{F}(z,x)=\int^x_0\hat{f}(z,s)ds,\hat{F}_{\pm}(z,x)=\int^x_0\hat{f}_{\pm}(z,s)ds$ and consider the $C^1$-functionals $\hat{\varphi},\hat{\varphi}_{\pm}:W^{1,p}(\Omega)\rightarrow\RR$ defined by
\begin{eqnarray*}
	&&\hat{\varphi}(u)=\frac{1}{p}||Du||^p_p+\frac{1}{p}\int_{\Omega}[\xi(z)+\hat{\xi}_0]|u|^pdz+\frac{1}{p}\int_{\partial\Omega}\beta(z)|u|^pd\sigma-\int_{\Omega}\hat{F}(z,u)dz,\\
	&&\hat{\varphi}_{\pm}(u)=\frac{1}{p}||Du||^p_p+\frac{1}{p}\int_{\Omega}[\xi(z)+\hat{\xi}_0]|u|^pdz+\frac{1}{p}\int_{\partial\Omega}\beta(z)|u|^pd\sigma-\int_{\Omega}\hat{F}_{\pm}(z,u)dz\\
	&&\mbox{for all}\ u\in W^{1,p}(\Omega).
\end{eqnarray*}

From (\ref{eq12}) and since $\hat{\xi}_0>0$, we get the following proposition.
\begin{prop}\label{prop4}
	If hypotheses $H(\xi),H(\beta),H(f)$ hold, then $\hat{\varphi}$ is even and coercive.
\end{prop}

From this proposition, we infer the following corollary.
\begin{corollary}\label{cor5}
	If hypotheses $H(\xi),H(\beta),H(f)$ hold, then the functional $\hat{\varphi}$ is bounded below and satisfies the PS-condition.
\end{corollary}

Also, we have this proposition.
\begin{prop}\label{prop6}
	If hypotheses $H(\xi),H(\beta),H(f)$ hold, then $K_{\hat{\varphi}}\subseteq C^1(\overline{\Omega})$ and there exists $M>0$ such that
	$$-M\leq u(z)\leq M\ \mbox{for all}\ z\in\overline{\Omega},\ \mbox{all}\ u\in K_{\hat{\varphi}}.$$
\end{prop}
\begin{proof}
	The inclusion $K_{\hat{\varphi}}\subseteq C^1(\overline{\Omega})$ follows from the nonlinear regularity theory (see Papageorgiou and R\u adulescu \cite{10} and Lieberman \cite{6}).
	
	From (\ref{eq12}), hypothesis $H(f)(i)$ and the fact that $\hat{\xi}_0>||\xi||_{\infty}$, we see that we can find $M>0$ such that
	\begin{equation}\label{eq13}
		|\hat{f}(z,x)|\leq[\xi(z)+\hat{\xi}_0]M^{p-1}\ \mbox{for almost all}\ z\in\Omega,\ \mbox{all}\ x\in\RR.
	\end{equation}
	
	Suppose that $u\in K_{\hat{\varphi}}$. Then for all $h\in W^{1,p}(\Omega)$, we have
	\begin{eqnarray}\label{eq14}
		&&\left\langle A(u),h\right\rangle+\int_{\Omega}[\xi(z)+\hat{\xi}_0]|u|^{p-2}uhdz+\int_{\partial\Omega}\beta(z)|u|^{p-2}uhd\sigma\nonumber\\
		&=&\int_{\Omega}\hat{f}(z,u)hdz\nonumber\\
		&\leq&\int_{\Omega}|\hat{f}(z,u)||h|dz\nonumber\\
		&\leq&\int_{\Omega}(\xi(z)+\hat{\xi}_0)M^{p-1}|h|dz\ (\mbox{see (\ref{eq13})})\nonumber\\
		&\leq&\int_{\Omega}[\xi(z)+\hat{\xi}_0]M^{p-1}|h|dz+\int_{\partial\Omega}\beta(z)M^{p-1}|h|d\sigma\ (\mbox{see hypothesis}\ H(\beta)).
	\end{eqnarray}
	
	In (\ref{eq14}) we choose $h=(u-M)^+\in W^{1,p}(\Omega)$. Then
	\begin{eqnarray*}
		&&\left\langle A(u),(u-M)^+\right\rangle+\int_{\Omega}[\xi(z)+\hat{\xi}_0]u^{p-1}(u-M)^+dz+\int_{\partial\Omega}\beta(z)u^{p-1}(u-M)^+d\sigma\\
		&\leq&\left\langle A(M),(u-M)^+\right\rangle+\int_{\Omega}[\xi(z)+\hat{\xi}_0]M^{p-1}(u-M)^+dz+\int_{\partial\Omega}\beta(z)M^{p-1}(u-M)^+d\sigma,\\
		&\Rightarrow&\left\langle A(u)-A(M),(u-M)^+\right\rangle+\int_{\Omega}[\xi(z)+\hat{\xi}_0](u^{p-1}-M^{p-1})(u-M)^+dz\leq 0\\
		&&(\mbox{note that}\ A(M)=0\ \mbox{and see hypothesis}\ H(\beta))\\
		&\Rightarrow&u\leq M.
	\end{eqnarray*}
	
	In a similar fashion, we can show that
	\begin{eqnarray*}
		&&-M\leq u,\\
		&\Rightarrow&u\in[-M,M]\cap C^1(\overline{\Omega})\ \mbox{for all}\ u\in K_{\hat{\varphi}}.
	\end{eqnarray*}
\end{proof}

We choose $\vartheta_0\geq\tilde{\vartheta}$ such that for almost all $z\in\Omega$, the functions
\begin{equation}\label{newrelation}
x\mapsto f(z,x)+\vartheta_0|x|^{p-2}x\ \mbox{and}\ x\mapsto(\vartheta+\vartheta_0)|x|^{p-2}x-c_1|x|^{r-2}x\end{equation}
are nondecreasing on $[-M,M]$ (see hypothesis $H(f)(iii)$).
\begin{prop}\label{prop7}
	If hypotheses $H(\xi),H(\beta),H(f)$ hold, then $\tilde{u}\leq u$ for all $u\in K_{\hat{\varphi}_+}\backslash\{0\}$ and $v\leq\tilde{v}$ for all $v\in K_{\hat{\varphi}_-}\backslash\{0\}$.
\end{prop}
\begin{proof}
	As before the nonlinear regularity theory and the nonlinear maximum principle imply that
	\begin{equation}\label{eq15}
		K_{\hat{\varphi}_+}\subseteq D_+\cup\{0\}\ \mbox{and}\ K_{\hat{\varphi}_-}\subseteq (-D_+)\cup\{0\}.
	\end{equation}
	
	Let $u\in K_{\hat{\varphi}_+},u\neq 0$. Let $t^*>0$ be the biggest real number such that
	\begin{equation}\label{eq16}
		t^*\tilde{u}\leq u\ (\mbox{see Filippakis and Papageorgiou \cite[Lemma 3.6]{2}}).
	\end{equation}
	
	Suppose that $t^*\in(0,1)$. Then for $\hat{\vartheta}_0=\hat{\xi}_0+\vartheta_0$ we have
	\begin{eqnarray}\label{eq17}
		&&-\Delta_p(t^*\tilde{u})(z)+[\xi(z)+\hat{\vartheta}_0](t^*\tilde{u})(z)^{p-1}\nonumber\\
		&=&(t^*)^{p-1}[-\Delta_p\tilde{u}(z)+[\xi(z)+\hat{\vartheta}_0]\tilde{u}(z)^{p-1}]\nonumber\\
		&\leq&(t^*)^{p-1}[-\Delta_p\tilde{u}(z)+[||\xi||_{\infty}+\hat{\vartheta}_0]\tilde{u}(z)^{p-1}]\nonumber\\
		&=&(t^*)^{p-1}[k(z,\tilde{u}(z))+\hat{\vartheta_0}\tilde{u}(z)^{p-1}] (\mbox{see Proposition \ref{prop3}})\nonumber\\
		&=&(t^*)^{p-1}[\vartheta\tilde{u}(z)^{p-1}-c_1\tilde{u}(z)^{r-1}+\hat{\vartheta}_0\tilde{u}(z)^{p-1}]\ (\mbox{see Proposition \ref{prop3} and (\ref{eq3})})\nonumber\\
		&<&\vartheta(t^*\tilde{u})(z)^{p-1}-c_1(t^*\tilde{u})(z)^{r-1}+\hat{\vartheta}_0(t^*\tilde{u})(z)^{p-1}\ \mbox{for almost all}\ z\in\Omega\\
		&&(\mbox{recall that}\ t^*<1, r>p).\nonumber
	\end{eqnarray}
	
	Let $\Omega_1=\{z\in\Omega:u(z)\leq\tilde{u}(z)\}$ and $\Omega_2=\{z\in\Omega:\tilde{u}(z)<u(z)\}$.
	
	For almost all $z\in\Omega_1$, we have
	\begin{eqnarray}\label{eq18}
		\hat{f}(z,u(z))+\vartheta_0u(z)^{p-1}&=&f(z,u(z))+\hat{\vartheta}_0u(z)^{p-1}\ (\mbox{see (\ref{eq12}) and (\ref{eq15})})\nonumber\\
		&>&\vartheta u(z)^{p-1}-c_1u(z)^{r-1}+\hat{\vartheta}_0u(z)^{p-1}\ (\mbox{see (\ref{eq2}) and recall that}\ \vartheta<\hat{\vartheta})\nonumber\\
		&\geq&\vartheta(t^*\tilde{u})(z)^{p-1}-c_1(t^*\tilde{u})(z)^{r-1}+
\hat{\vartheta}_0(t^*\tilde{u})(z)^{p-1}.
	\end{eqnarray}
For the last inequality we have used relation \eqref{eq16} in combination with the monotonicity of the mapping $x\mapsto(\vartheta+\vartheta_0)|x|^{p-2}x-c_1|x|^{r-2}x$, see \eqref{newrelation}.
	
	For almost all $z\in\Omega_2$, we have
	\begin{eqnarray}\label{eq19}
		\hat{f}(z,u(z))+\vartheta_0u(z)^{p-1}&\geq&\hat{f}(z,u(z))+\vartheta_0\tilde{u}(z)^{p-1}\ (\mbox{since}\ z\in\Omega_2)\nonumber\\
		&=&f(z,\tilde{u}(z))+\hat{\vartheta}_0\tilde{u}(z)^{p-1}\ (\mbox{see (\ref{eq12})})\nonumber\\
		&\geq&f(z,t^*\tilde{u}(z))+\hat{\vartheta}_0(t^*\tilde{u})(z)^{p-1}\nonumber\\
		&& (\mbox{recall that we have assumed}\ t^*<1)\nonumber\\
		&>&\vartheta(t^*\tilde{u})(z)^{p-1}-c_1(t^*\tilde{u})(z)^{r-1}+\hat{\vartheta}_0(t^*\tilde{u})(z)^{p-1}\\
		&&(\mbox{see (\ref{eq2}), \eqref{newrelation} and recall that}\ \vartheta<\hat{\vartheta})\nonumber.
	\end{eqnarray}
	
	Returning to (\ref{eq17}) and using (\ref{eq18}) and (\ref{eq19}) we see that
	\begin{eqnarray}\label{eq20}
		&&-\Delta_p(t^*\tilde{u})(z)+[\xi(z)+\hat{\vartheta}_0](t^*\tilde{u})(z)^{p-1}\nonumber\\
		&<&\hat{f}(z,u(z))+\vartheta_0u(z)^{p-1}\nonumber\\
		&=&-\Delta_pu(z)+[\xi(z)+\hat{\vartheta}_0]u(z)^{p-1}\ \mbox{for almost all}\ z\in\Omega\ (\mbox{recall that}\ u\in K_{\hat{\varphi}_+}).
	\end{eqnarray}
	
	We introduce the following functions
	\begin{eqnarray*}
		&&h_1(z)=(\vartheta+\hat{\vartheta}_0)(t^*\tilde{u})(z)^{p-1}-c_1(t^*\tilde{u})(z)^{r-1},\\
		&&h_2(z)=(\hat{\vartheta}+\hat{\vartheta}_0)(t^*\tilde{u})(z)^{p-1}-c_1(t^*\tilde{u})(z)^{r-1},\\
		&&h_3(z)=\hat{f}(z,u(z))+\hat{\vartheta}_0u(z)^{p-1}.
	\end{eqnarray*}
	
	Evidently $h_1,h_2\in C^1(\overline{\Omega})$ and $h_3\in L^{\infty}(\Omega)$ (see Proposition \ref{prop6} and (\ref{eq12})). We have
	\begin{eqnarray}\label{eq21}
		&&h_1(z)<h_2(z)\ \mbox{for all}\ z\in\overline{\Omega}\ (\mbox{recall that}\ \vartheta<\hat{\vartheta}),\nonumber\\
		&\Rightarrow&h_1\prec h_2.
	\end{eqnarray}
	
	Also, we have
	\begin{eqnarray*}
		&&h_2\leq h_3\ (\mbox{see (\ref{eq2})}),\\
		&\Rightarrow&h_1\prec h_3\ (\mbox{see (\ref{eq21})}).
	\end{eqnarray*}
	
	By virtue of (\ref{eq20}) we can use Proposition \ref{prop1} (the strong comparison principle) and have
	$$u-t^*\tilde{u}\in\hat{D}_+,$$
	which contradicts the maximality of $t^*>0$ (recall that $\hat{D}_+\subseteq C^1(\overline{\Omega})$ is open). So, we have
	\begin{eqnarray*}
		&&1\leq t^*,\\
		&\Rightarrow&\tilde{u}\leq u\ \mbox{for all}\ u\in K_{\hat{\varphi}_+}\ (\mbox{see (\ref{eq16})}).
	\end{eqnarray*}
	
	In a similar fashion we show that
	$$v\leq\tilde{v}\ \mbox{for all}\ v\in K_{\hat{\varphi}_-}.$$
\end{proof}
\begin{prop}\label{prop8}
	If hypotheses $H(\xi), H(\beta),H(f)$ hold and $V\subseteq W^{1,p}(\Omega)$ is a nontrivial finite dimensional subspace, then we can find $\rho_V>0$ such that
	$$\sup[\hat{\varphi}(u):u\in V, ||u||=\rho_V]<0.$$
\end{prop}
\begin{proof}
	
	Since $V$ is finite dimensional, all norms are equivalent. Hence, we can find $\rho_V>0$ such that
	\begin{equation}\label{eq23}
		u\in V,\ ||u||\leq\rho_V\Rightarrow|u(z)|\leq\delta\ \mbox{for almost all}\ z\in\Omega.
	\end{equation}
	
	Then for $u\in V$ with $||u||=\rho_V$, we have
	$$\hat{\varphi}(u)=\frac{1}{p}||Du||^p_p+\frac{1}{p}\int_{\Omega}[\xi(z)+\hat{\xi}_0]|u|^pdz+\frac{1}{p}
\int_{\partial\Omega}\beta(z)|u|^pd\sigma-\int_{\Omega}\hat{F}(z,u)dz.$$
	
	Clearly, we can always have
	$$\delta\in\left(0, \min\left\{{\underset{\mathrm{\overline{\Omega}}}\min}\ \tilde{u}, {\underset{\mathrm{\overline{\Omega}}}\min}(-\tilde{v})\right\}\right)\ (\mbox{recall that}\ \tilde{u}\in D_+,\tilde{v}\in-D_+).$$
	
	So, by (\ref{eq7}), (\ref{eq12}), (\ref{eq23}),  we have
	\begin{eqnarray*}
		\hat{\varphi}(u)&\leq&\frac{1}{p}||Du||^p_p+\frac{||\xi||_{\infty}}{p}||u||^p_p+\frac{1}{p}\int_{\partial\Omega}\beta(z)|u|^pd\sigma-\frac{\mu}{p}||u||^p_p\\
		&\leq&\frac{1}{p}[c_3-\mu c_4]||u||^p\ \mbox{for some}\ c_3,c_4>0\\
		&&(\mbox{here we have used the fact that on $V$ all norms are equivalent}).
	\end{eqnarray*}
	
	Choosing $\mu>\frac{c_3}{c_4}$, we see that
	$$\hat{\varphi}(u)<0\ \mbox{for all}\ u\in V\ \mbox{such that}\ ||u||=\rho_V.$$
\end{proof}

Now we are ready for the main result of this work.
\begin{theorem}\label{th9}
	If hypotheses $H(\xi),H(\beta),H(f)$ hold, then problem (\ref{eq1}) admits a  sequence of distinct nodal solutions
	$$\{u_n\}_{n\geq 1}\subseteq C^1(\overline{\Omega})\ \mbox{such that}\ u_n\rightarrow 0\ \mbox{on}\ C^1(\overline{\Omega}).$$
\end{theorem}
\begin{proof}
	Corollary \ref{cor5} and Proposition \ref{prop8} permit us to use  Theorem \ref{th2}. So, we can find $\{u_n\}_{n\geq 1}\subseteq W^{1,p}(\Omega)$ such that
	\begin{equation}\label{eq24}
		u_n\in K_{\hat{\varphi}}\ \mbox{for all}\ n\in\NN\ \mbox{and}\ u_n\rightarrow 0\ \mbox{in}\ W^{1,p}(\Omega).
	\end{equation}
	
	From \eqref{eq23} and Proposition \ref{prop6} we have that $u_n\in C^1(\overline{\Omega})$ for all $n\in\NN$. Moreover, Proposition \ref{prop4} together with Theorem \ref{th2} (see also \cite{6}) imply that we can find $\alpha'\in(0,1)$ and $c_5>0$ such that
	\begin{equation}\label{eq25}
		u_n\in C^{1,\alpha'}(\overline{\Omega})\ \mbox{and}\ ||u_n||_{C^{1,\alpha'}(\overline{\Omega})}\leq c_5\ \mbox{for all}\ n\in\NN.
	\end{equation}
	
	Exploiting the compact embedding of $C^{1,\alpha}(\overline{\Omega})$ into $C^1(\overline{\Omega})$, we can infer from  (\ref{eq24}) and (\ref{eq25}) that
	$$u_n\rightarrow 0\ \mbox{in}\ C^1(\overline{\Omega})\ \mbox{as}\ n\rightarrow\infty.$$
	
	So, we can find $n_0\in\NN$ such that
	$$u_n\in[\tilde{v},\tilde{u}]\ \mbox{for all}\ n\geq n_0.$$
	
	Since $\tilde{v},\, \tilde{u}$ are not solutions of (\ref{eq1}) (see (\ref{eq2}), (\ref{eq3}) and recall that $\vartheta<\hat{\vartheta})$, on account of Proposition \ref{prop7} and (\ref{eq12}), we see that $\{u_n\}_{n\geq n_0}\subseteq C^1(\overline{\Omega})$ (by the nonlinear regularity theory) are nodal solutions of problem (\ref{eq1}).
\end{proof}

 \medskip
{\bf Acknowledgements.} This research was supported by the Slovenian Research Agency grants P1-0292, J1-8131, J1-7025. V.D. R\u adulescu acknowledges the support through a grant of the Romanian National Authority for Scientific Research and Innovation, CNCS-UEFISCDI, project number PN-III-P4-ID-PCE-2016-0130.

\end{document}